\documentclass[12pt]{amsart}

\usepackage{amssymb}
\usepackage{verbatim}
\usepackage[toc,page]{appendix}

\newtheorem{thm}{Theorem}[section]

\newtheorem{lem}[thm]{Lemma}
\newtheorem{cor}[thm]{Corollary}

\theoremstyle{definition}

\theoremstyle{remark}

\numberwithin{equation}{section}

\newcommand{\R}{\mathbf{R}}  

\newcommand{\C}{\mathbf{C}}
\newcommand{\Z}{\mathbf{Z}}

\begin{document}

\title[The first moment of cusp form $L$-functions]{The first moment of cusp form $L$-functions in weight aspect on average}

\author{Olga  Balkanova}
\address{University of Turku, Department of Mathematics and Statistics, Turku, 20014, Finland}
\email{olgabalkanova@gmail.com}
\thanks{Research of the first author is supported by Academy of Finland project no. $293876$.}

\author{Dmitry  Frolenkov}
\address{National Research University Higher School of Economics, Moscow, Russia and Steklov Mathematical Institute of Russian Academy of Sciences,  8 Gubkina st., Moscow, 119991, Russia}
\email{frolenkov@mi.ras.ru}
\begin{abstract}
We study the asymptotic behaviour  of the twisted first moment of central $L$-values associated to cusp forms in weight aspect on average. Our estimate of the error term allows extending the logarithmic length of mollifier $\Delta$ up to $2$. The best previously known result, due to Iwaniec and Sarnak, was $\Delta<1$.
The proof is based on a representation formula for the error  in terms of Legendre polynomials.
\end{abstract}


\maketitle

\tableofcontents


\section{Introduction}
The  technique of mollification allows proving strictly positive non-vanishing results for different families of $L$-functions.  The idea of the method is to regularize the behavior of $L$-functions while averaging over the family by  introducing smoothing weights called mollifiers. Common choice for mollifier is a Dirichlet polynomial of length $M$ approximating the inverse of $L$-function.
It is of crucial importance to optimize the parameter $M$, called mollifier's length, as it determines the proportion of non-vanishing $L$-values.

Consider the family $H_{2k}(1)$ of primitive forms of level $1$ and weight $2k \geq 12$.
Every $f \in H_{2k}(1)$ has a Fourier expansion of the form
\begin{equation}
f(z)=\sum_{n\geq 1}\lambda_f(n)n^{(2k-1)/2}e(nz).
\end{equation}
The associated $L$-function is defined by
\begin{equation}L_f(s)=\sum_{n \geq 1}\frac{\lambda_f(n)}{n^s}, \quad \Re{s}>1.
\end{equation}

Let $\Gamma(s)$ be the Gamma function. The completed $L$-function
\begin{equation}
\Lambda_f(s)=\left(\frac{1}{2 \pi}\right)^s\Gamma\left(s+\frac{2k-1}{2}\right)L_f(s)
\end{equation}
 satisfies the functional equation
\begin{equation} \label{eq: functionalE}
\Lambda_f(s)=\epsilon_f\Lambda_f(1-s), \quad \epsilon_f=i^{2k}
\end{equation}
and can be analytically continued on the whole complex plane.

 The harmonic summation is defined by
\begin{equation}
\sum_{f \in H_{2k}(1)}^{h}\alpha_f:=\sum_{f \in H_{2k}(1)}\alpha_f\frac{ \Gamma(2k-1)}{(4\pi)^{2k-1}  \langle f,f\rangle_1},
\end{equation}
where $\langle f,f\rangle_1$ is the Petersson inner product on the space of level $1$ holomorphic modular forms.

The usual choice for mollifier is
\begin{equation}
M(f)=\sum_{m\leq M}\frac{x_m\lambda_f(m)}{m^{1/2}}, \quad x_m \in \R,
\end{equation}
where parameter $M$ is called the length of mollifier.

Let  $h$ be a suitable test function (see section \ref{averagingoverweight} for details) and
\begin{equation}
H:=\int_{0}^{\infty}h(y)dy.
\end{equation}

Let $\mu(m)$ be the M\"{o}bius function and $\sigma(m)$ be the sum of divisors function.
Iwaniec and Sarnak \cite[Theorem~3]{IS} proved that for the mollifier with
\begin{equation} \label{eq:xm}
x_m \sim \frac{\mu(m)m (\log{M}/\log{m})^2}{\sigma(m)2\zeta(2)\log{M}}
\end{equation}
of length $M\leq K(\log{K})^{-20}$, one has
\begin{equation}\label{eq:firstm}
\sum_{k}h\left( \frac{2k}{K}\right)\sum_{f \in H_{2k}(1)}^{h}L_f(1/2)M(f)\sim HK,
\end{equation}
\begin{equation}\label{eq:secondm}
\sum_{k}h\left( \frac{2k}{K}\right)\sum_{f \in H_{2k}(1)}^{h}L_{f}^{2}(1/2)M^2(f)\sim 2HK\left( 1+\frac{\log{K}}{\log{M}}\right).
\end{equation}

Let $M:=K^{\Delta}$, where parameter $\Delta$ is called the logarithmic length of mollifier.
Note that if $\epsilon_f=i^{2k}=-1$, then it follows from functional equation \eqref{eq: functionalE} that $ L_{f}(1/2)$ is identically zero. For $\epsilon_f=1$
equations \eqref{eq:firstm}, \eqref{eq:secondm}
imply (see \cite{BF1} for details) that  at least
\begin{equation}\label{eq:proportion}
\frac{\Delta}{\Delta+1}
\end{equation} of central $L$-values do not vanish on average as $K\rightarrow \infty$. 

Taking the largest admissible $\Delta=1-\epsilon$, the percentage of non-vanishing is no less than $50\%$.
Furthermore, according to \cite[Proposition~16]{IS} any improvement over $50\%$ with an additional lower bound on $L_f(1/2)$ would imply the non-existence of Landau-Siegel zeros for Dirichlet $L$-functions of real primitive characters.

In order to break the $50\%$ barrier one needs to increase the length of mollifier for both first  and second moments.  

In the present paper we consider only the first moment and show that  equation \eqref{eq:firstm} holds for the length of mollifier $M\leq K^{2-\epsilon}$ for any $\epsilon>0$.
This extension follows from the asymptotic formula for the twisted first moment.
\begin{thm}\label{thm:main}
For all $l $ one has
\begin{multline}
M_1(l):=\sum_{k}h\left(\frac{4k}{K} \right)\sum_{f \in H_{4k}(1)}^h\lambda_f(l)L_f(1/2)=\\ \frac{2}{\sqrt{l}}\frac{HK}{4}+O\left(K\frac{l^{1/2+\epsilon}}{K^2}\right).
\end{multline}
\end{thm}

More precisely, the mollified moment can be expressed in terms of the twisted moment
\begin{equation*}
\sum_{k}h\left( \frac{4k}{K}\right)\sum_{f \in H_{4k}(1)}^{h}L_f(1/2)M(f)=
\sum_{m\leq M}\frac{x_m}{\sqrt{m}}M_1(m).
\end{equation*}

Then the length of mollifier  is the largest admissible $M$ such that
\begin{equation*}
K\sum_{m\leq M}\frac{|x_m|}{\sqrt{m}}\frac{m^{1/2+\epsilon}}{K^2}\ll K^{1-\epsilon}.
\end{equation*}
Therefore, for any mollifier with $|x_m| \leq \log{M}$ and any $\epsilon>0$  one can take  $M\leq K^{2-\epsilon}$. Consequently, the logarithmic length of mollifier $\Delta$ can be extended up to  $2$.

The detailed description of the mollifier method and analogous results for an individual weight can be found in \cite{BF1}.
\section{Special functions}

For $z \in \C$, $\Re{z}>0$ the Gamma function is defined by
\begin{equation}
\Gamma(z)=\int_{0}^{\infty}t^{z-1}\exp(-t)dt.
\end{equation}
By \cite[Eq.~5.5.5]{HMF} for $2z\neq 0,-1,-2,\ldots $ one has
\begin{equation}\label{eq:gamma2}
\Gamma(2z)=\frac{1}{\sqrt{\pi}}2^{2z-1}\Gamma(z)\Gamma(z+1/2).
\end{equation}

Let $$e(x):=\exp{(2\pi ix)}.$$
The confluent hypergeometric function
\begin{equation}
{}_1F_{1}(a,b;x) =\frac{\Gamma(b)}{\Gamma(a)}\sum_{k=0}^{\infty}\frac{\Gamma(a+k)}{\Gamma(b+k)}\frac{x^{k}}{k!}
\end{equation}
 can be expressed in terms of the Bessel function of the first kind
\begin{equation}
J_{v}(x)=\sum_{m=0}^{\infty}\frac{(-1)^m}{m!\Gamma(m+v+1)}\left( \frac{x}{2}\right)^{2m+v}.
\end{equation}
\begin{lem} For $\epsilon=\pm 1$ one has
\begin{multline}\label{eq:repF}
{}_1F_{1}\left(k,2k;2z\right)=\Gamma(k+1/2)\exp(z)\left(\frac{z}{2}\right)^{1/2-k}\times \\ e\left(\epsilon\frac{1/2-k}{4}\right)J_{k-1/2}\left( ze\left(\frac{\epsilon}{4}\right)\right).
\end{multline}
\end{lem}
\begin{proof}
Using \cite[Eq.~13.2.2]{HMF} and \cite[Eq.~13.6.9]{HMF}, we write the confluent hypergeometric function
in terms of the $I$-Bessel function. Further, applying \cite[Eq.~10.27.6]{HMF}, we prove the required result.
\end{proof}

Legendre polynomials are $n$th degree polynomials given by Rodrigues' formula
\begin{equation}
P_n(x)=\frac{1}{2^n n!}\frac{d^n}{dx^n}\left[(x^2-1)^n\right].
\end{equation}
Note that by \cite[Eq.~14.7.17]{HMF} 
\begin{equation}\label{eq:legsign}
P_n(-x)=(-1)^nP_n(x).
\end{equation}

\begin{lem} For any nonnegative integer $n$ one has
\begin{equation}\label{eq:LegendreRepr}
J_{n+1/2}(z)=(-i)^n\sqrt{\frac{z}{2\pi}}\int_{0}^{\pi}\exp(iz\cos{\theta})P_n(\cos{\theta})\sin{\theta}d\theta.
\end{equation}
\end{lem}
\begin{proof}
The assertion follows from  \cite[Eq.~10.47.3]{HMF} and \cite[Eq.~10.54.2]{HMF}.
\end{proof}

\begin{lem} Let $|\arg{z}|<\pi$. As $z \rightarrow \infty$ we have
\begin{multline}\label{eq:repJBessel}
J_0(z)=\sqrt{\frac{2}{\pi z}}\Biggl(
\cos(z-\pi/4)\left[ \sum_{j=0}^{d-1}(-1)^s\frac{a_{2j}}{z^{2j}}+R_1\right]-\\
\sin(z-\pi/4)\left[ \sum_{j=0}^{d-1}(-1)^s\frac{a_{2j+1}}{z^{2j+1}}+R_2\right]
\Biggr),
\end{multline}
where
\begin{equation}
a_{j}=\frac{\Gamma(j+1/2)}{2^jj!\Gamma(-j+1/2)} \text{ for } j \geq 0,
\end{equation}
\begin{equation}
R_1=O\left(\frac{1}{(2z)^{2d}}\right),\quad R_2=O\left(\frac{1}{(2z)^{2d+1}}\right).
\end{equation}
\end{lem}
\begin{proof}
See \cite[Eq.~8.451.1]{GR} and \cite[Eq.~10.17.3]{HMF}.
\end{proof}

\section{Exact formula for the first moment}\label{sec:EF}
In this section we consider the first moment of primitive $L$-functions and show how to express the error in terms of special functions.
For $\epsilon_1=\pm 1$ let us define
\begin{multline}\label{eq:I}
I_{\epsilon_1}(u,v,k;x):=e\left(\frac{\epsilon_1}{8}-\frac{\epsilon_1 k}{4}\right)x^{1/2-k}\frac{\Gamma(k-v-u)} {\Gamma(2k)}\\ \times{}_1F_{1}\left(k-v-u,2k,-\frac{e(-\epsilon_1/4)}{x}\right).
\end{multline}

As a consequence of the Petersson trace formula we obtain the exact formula for the twisted first moment.
\begin{thm}\label{Thm:explicitformula}
For $2k \geq 12$, $\Re{v}=0$,  $|\Re{u}|<k-1$ we have
\begin{multline}\label{eq:exactf}
\sum_{f \in H_{2k}(1)}^{h}\lambda_f(l)L_f(1/2+u+v)=\frac{1}{l^{1/2+u+v}}\\+i^{2k}\frac{(2\pi)^{2u+2v}\Gamma(k-u-v)}{l^{1/2-u-v}\Gamma(k+u+v)}
+2\pi i^{2k}V_1(l;u,v,k).
\end{multline}
The error term is given by
\begin{multline}\label{eq:error}
V_1(l;u,v,k)=\sum_{c=1}^{\infty}\frac{1}{c^{1/2+u+v}}\sum_{\substack{n=1\\ (n,c)=1}}^{\infty}\frac{e(n^{*}lc^{-1})}{n^{1/2-u-v}}(2\pi)^{u+v-1/2}\\
\times e \left(\frac{ 1/2-u-v}{4}\right)I_{-1}\left(u,v,k;\frac{cn}{2\pi l}\right)+\sum_{c=1}^{\infty}\frac{1}{c^{1/2+u+v}}\sum_{\substack{n=1\\ (n,c)=1}}^{\infty}\frac{e(-n^{*}lc^{-1})}{n^{1/2-u-v}}\\
\times (2\pi)^{u+v-1/2}e \left(-\frac{ 1/2-u-v}{4}\right)I_{+1}\left(u,v,k;\frac{cn}{2\pi l}\right),
\end{multline}
where $nn^*\equiv 1 \pmod{c}$.
\end{thm}
\begin{proof}
This formula was proved in \cite[Sections~4-5]{BF} for prime power level $N=p^{v}$, $p$ prime, $v\geq 2$. 
When the level $N$ is equal to $1$, the function under the integral in \cite[Eq.~4.16]{BF} has a pole in view of \cite[Eq.~4.15]{BF}. Consequently, we cross this pole while shifting the contour of integration in the proof of \cite[Lemma~4.8]{BF}. This yields the additional main term
\begin{equation*}
i^{2k}\frac{(2\pi)^{2u+2v}\Gamma(k-u-v)}{l^{1/2-u-v}\Gamma(k+u+v)}
\end{equation*}
in formula \eqref{eq:exactf}.
The rest of the proof is exactly the same.
\end{proof}

We are interested in the behavior of the first moment at the critical point $1/2$ and therefore can let $u=v=0$.

\begin{lem} For $\epsilon_1=\pm 1$ one has
\begin{equation}\label{besselrepresentation}
e(-\epsilon_1/8)I_{\epsilon_1}(0,0,k;x)=\sqrt{\pi}
e\left(\frac{\epsilon_1 }{4\pi x}\right)e\left( \frac{-\epsilon_1k}{4}\right)J_{k-1/2}\left( \frac{1}{2x}\right).
\end{equation}
\end{lem}
\begin{proof}
Substituting representation \eqref{eq:repF} in equation \eqref{eq:I} we have
\begin{multline*}
e(-\epsilon_1/8)I_{\epsilon_1}(0,0,k;x)=e\left( \frac{-\epsilon_1k}{4}\right)x^{1/2-k}\frac{\Gamma(k)\Gamma(k+1/2)}{\Gamma(2k)}\times\\
2^{2k-1}\exp\left(-\frac{e(-\epsilon_1/4)}{2x} \right)\left( -\frac{e(-\epsilon_1/4)}{x} \right)^{1/2-k}\times \\
e\left( \epsilon_2\frac{1/2-k}{4}\right)J_{k-1/2}\left( -\frac{e(-\epsilon_1/4)e(\epsilon_2/4)}{2x}\right),
\end{multline*}
where $\epsilon_2=\pm 1$.
Note that $-e(-\epsilon_1/4)=\epsilon_1 i$.
Choosing $\epsilon_2=-\epsilon_1$ yields
\begin{equation*}
-e(-\epsilon_1/4)e(\epsilon_2/4)=-\exp(\pi i)=1.
\end{equation*}
Thus
\begin{multline*}
e(-\epsilon_1/8)I_{\epsilon_1}(0,0,k;x)=\\\frac{\Gamma(k)\Gamma(k+1/2)}{\Gamma(2k)}2^{2k-1}
e\left(\frac{\epsilon_1 }{4\pi x}\right)e\left( \frac{-\epsilon_1k}{4}\right)J_{k-1/2}\left( \frac{1}{2x}\right).
\end{multline*}
It follows by equation \eqref{eq:gamma2} that
\begin{equation*}
e(-\epsilon_1/8)I_{\epsilon_1}(0,0,k;x)=\sqrt{\pi}
e\left(\frac{\epsilon_1 }{4\pi x}\right)e\left( \frac{-\epsilon_1k}{4}\right)J_{k-1/2}\left( \frac{1}{2x}\right).
\end{equation*}
\end{proof}
\begin{cor}\label{cor:error}
For $\epsilon_1=\pm 1$ one has
\begin{equation}
I_{\epsilon_1}(0,0,k;x)\ll \frac{(4x)^{-k+1/2}}{\Gamma(k+1/2)}.
\end{equation}
\end{cor}
\begin{proof}
By formula \cite[Eq.~10.14.4]{HMF}
\begin{equation*}
|J_{k-1/2}(z)|\leq \frac{(z/2)^{k-1/2}}{\Gamma(k+1/2)}.
\end{equation*}
Taking $z=1/(2x)$ we prove the assertion.
\end{proof}

 Furthermore, function $I_{\epsilon_1}$ has an integral representation in terms of Legendre polynomials.
\begin{lem}\label{lem:legendre} For $k \equiv 0 \pmod{2}$, $\epsilon_1=\pm 1$ one has
\begin{multline}
I_{\epsilon_1}\left(0,0,k;\frac{1}{2z}\right)=-e\left(\frac{\epsilon_1}{8}\right) e\left(\frac{\epsilon_1z}{2\pi}\right)\sqrt{2z}\times \\
\int_{0}^{\pi/2}\sin\left(z\cos{\theta}\right)P_{k-1}(\cos{\theta})\sin{\theta}d\theta.
\end{multline}
\end{lem}
\begin{proof}
Consider representation \eqref{besselrepresentation} with $z:=(2x)^{-1}$. Applying \eqref{eq:LegendreRepr} with $n=k-1$, we obtain
\begin{multline*}
e(-\epsilon_1/8)I_{\epsilon_1}\left(0,0,k;\frac{1}{2z}\right)=
e\left( \frac{-\epsilon_1k}{4}\right)e\left( -\frac{k}{4}+\frac{1}{4}\right)e\left( \frac{\epsilon_1 z}{2\pi}\right)\times \\
\sqrt{\frac{z}{2}}\int_{0}^{\pi}\exp(iz\cos{\theta})P_{k-1}(\cos{\theta})\sin{\theta}d\theta.
\end{multline*}
Since $k$ is an even integer, one has $e(-\epsilon_1k/4)e(-k/4)=1$ and
\begin{multline*}
e(-\epsilon_1/8)I_{\epsilon_1}\left(0,0,k;\frac{1}{2z}\right)=
e(1/4)e\left( \frac{\epsilon_1 z}{2\pi}\right)\times \\
\sqrt{\frac{z}{2}}\int_{0}^{\pi}\exp(iz\cos{\theta})P_{k-1}(\cos{\theta})\sin{\theta}d\theta.
\end{multline*}
Now we split the integral into two parts
$\int_{0}^{\pi}=\int_{0}^{\pi/2}+\int_{\pi/2}^{\pi}$ and make the change of variables
$\phi:=\pi-\theta$ in the second integral. Property \eqref{eq:legsign} yields that
\begin{equation*}
P_{k-1}(-\cos{\phi})=(-1)^{k-1}P_{k-1}(\cos{\phi})=-P_{k-1}(\cos{\phi})
\end{equation*}
for even  $k$. Finally, since $e(1/4)=\exp(\pi i/2)=i$ we have
\begin{multline*}
e(1/4)\int_{0}^{\pi}\exp(iz\cos{\theta})P_{k-1}(\cos{\theta})\sin{\theta}d\theta=\\i\int_{0}^{\pi/2}P_{k-1}(\cos{\theta})\sin{\theta} \left[\exp(iz\cos{\theta})-\exp(-iz\cos{\theta}) \right] d\theta=\\
-2\int_{0}^{\pi/2}\sin(z\cos{\theta})P_{k-1}(\cos{\theta})\sin{\theta}d\theta.
\end{multline*}
The assertion follows.
\end{proof}


\section{Asymptotic approximation of Legendre polynomials}

The following theorem is obtained by taking $\alpha=\beta=0$ in \cite[Eq.~1.1-1.3]{BG}.
\begin{thm}\label{thm:BG} (Baratella, Gatteschi, \cite{BG})
Let $N:=n+1/2$. Then
\begin{multline}
P_n(\cos{\theta})=\sqrt{\frac{\theta}{\sin{\theta}}} \Biggl(J_0(N\theta)\sum_{s=0}^{m}
\frac{A_s(\theta)}{N^{2s}}+\\\theta J_1(N\theta)\sum_{s=0}^{m-1}\frac{B_s(\theta)}{N^{2s+1}}+E_m
\Biggr),
\end{multline}
where for fixed positive constants $c$ and $\delta$ one has
\begin{equation}
E_m\ll \begin{cases}
\theta^{1/2}N^{-2m-3/2}& \text{ if }c/N\leq \theta\leq \pi-\delta\\
\theta^2N^{-2m}& \text{ if }0<\theta\leq c/N
\end{cases}.
\end{equation}
The functions $A_s(\theta)$, $B_s(\theta)$ are analytic for $0\leq \theta<\pi$ and defined recursively, starting from $A_0(\theta)=1$, by
\begin{equation}\label{eq:B}
\theta B_s(\theta)=-\frac{1}{2}A_{s}'(\theta)-\frac{1}{2}\int_{0}^{\theta}\frac{A_s'(t)}{t}dt
+\frac{1}{2}\int_{0}^{\theta}f(t)A_s(t)dt,
\end{equation}
\begin{equation}\label{eq:A}
A_{s+1}(\theta)=\frac{1}{2}\theta B_s'(\theta)-\frac{1}{2}\int_{0}^{\theta}tf(t)B_s(t)dt+\lambda_{s+1},
\end{equation}
with
\begin{equation}
f(t)=\frac{1}{4t^2}-\frac{1}{16\sin^2(t/2)}-\frac{1}{16\cos^2{(t/2)}},
\end{equation}
where the constants of integration $\lambda_{s+1}$ are chosen such that $A_{s+1}(0)=0$ for any integral $s\geq 0$.
\end{thm}


\section{Averaging over weight}\label{averagingoverweight}
Let $h \in C_{0}^{\infty}(\R^{+})$ be a  non-negative function, compactly supported on interval $[\theta_1,\theta_2]$ such that $\theta_2>\theta_1>0$ and
\begin{equation}
\|h^{(n)}\|_1\ll 1 \text{ for all }n \geq 0.
\end{equation}

Applying the Poisson summation and integrating by parts $a \geq 2$ times, we have
\begin{equation}
\sum_{k}h\left(\frac{4k}{K} \right)=\frac{HK}{4}+O\left( \frac{1}{K^a}\right),
\end{equation}
where
$$H=\int_{0}^{\infty}h(y)dy.$$
In this section we prove theorem \ref{thm:main}. Namely
we show that for all $l $ one has
\begin{equation}\label{mainasympt}
\sum_{k}h\left(\frac{4k}{K} \right)\sum_{f \in H_{4k}(1)}^h\lambda_f(l)L_f(1/2)=\frac{2}{\sqrt{l}}\frac{HK}{4}+O\left(K\frac{l^{1/2+\epsilon}}{K^2}\right).
\end{equation}

The main term in \eqref{mainasympt} is obtained by taking $u=v=0$ in theorem \ref{Thm:explicitformula} and averaging the main terms with respect to $k$.
Note that in formula \eqref{mainasympt} the summation is over elements of weight $4k$, and therefore, Theorem \ref{Thm:explicitformula} should be used with $k$ replaced by $2k$. The same applies to all other results of Section \ref{sec:EF}.

Consider \eqref{eq:error} with $u=v=0$. Let $z:=\pi l/(cn)$. We split the error term into two parts
\begin{equation*}
V_1(l;0,0,2k)=W_1(l,k)+W_2(l,k),
\end{equation*}
where the summation in $W_1(l,k)$ is over $c,n$ such that $z<k/5$ and in $W_2(l,k)$ such that $z\geq k/5$.
\begin{lem}\label{lem:w1}
There exists an absolute constant $C>1$ such that
\begin{equation}
\sum_{k}h\left(\frac{4k}{K} \right) W_1(l,k)\ll  \frac{l^{1/2+\epsilon}K^{\epsilon}}{C^K}.
\end{equation}

\end{lem}
\begin{proof}
 Let $d:=cn$. Since $z<k/5$ one has $d>d_0$ with $d_0:=5\pi l/k$.
By corollary \ref{cor:error}
\begin{multline*}
W_1(l,k)\ll \frac{1}{\Gamma(2k+1/2)}\sum_{d>d_0}d^{-1/2+\epsilon}
\left( \frac{\pi l}{2d}\right)^{2k-1/2}\ll\\
\frac{e^{2k}}{(2k)^{2k}}\left( \frac{\pi l}{2}\right)^{2k-1/2}\int_{d_0}^{\infty}
x^{-2k+\epsilon}dx \ll \\  d_0^{1/2+\epsilon}\frac{e^{2k}}{(2k)^{2k}}\left(\frac{\pi l}{2d_0} \right)^{2k-1/2}\ll l^{1/2+\epsilon}k^{-1}\left(\frac{e}{20}\right)^{2k}.
\end{multline*}

Summing the result over $k$ with the test function, we prove the assertion.
\end{proof}
If $l \ll K$ with a sufficiently small implied constant, the sums over $c$ and $n$ in $W_2(l,k)$ are empty and the error term in \eqref{mainasympt} can be estimated using lemma \ref{lem:w1}. Otherwise, the main contribution comes from the term involving $W_2(l,k)$, as we now show.

\begin{lem}\label{lem:w2} For any $\epsilon>0$ one has
\begin{equation}
\sum_{k}h\left(\frac{4k}{K} \right) W_2(l,k)\ll K\frac{l^{1/2+\epsilon}}{K^2}.
\end{equation}
\end{lem}
\begin{proof}
 It follows from lemma \ref{lem:legendre} that
\begin{multline*}
\sum_{k}h\left(\frac{4k}{K} \right)W_2(l,k)\ll \sum_{cn\ll l/k}\frac{\sqrt{l}}{cn} \int_{0}^{\pi/2}
\left| \sum_{k}h\left(\frac{4k}{K}\right) P_{2k-1}(\cos{\theta}) \right|\times \\ \sin{\theta}d\theta\ll
l^{1/2+\epsilon}\int_{0}^{\pi/2}
\left| \sum_{k}h\left(\frac{4k}{K}\right) P_{2k-1}(\cos{\theta}) \right|\sin{\theta}d\theta.
\end{multline*}

To approximate the Legendre polynomials we apply theorem \ref{thm:BG} with $m=1$ and $N=2k-1/2$, i.e.
\begin{multline*}
P_{2k-1}(\cos{\theta})=\sqrt{\frac{\theta}{\sin{\theta}}}\Biggl( J_0(N\theta)\left(1+\frac{A_1(\theta)}{N^2}\right)+ \theta J_1(N\theta)\frac{B_0(\theta)}{N}+E_1 \Biggr),
\end{multline*}
where $B_0(\theta)$ and $A_1(\theta)$ are defined by \eqref{eq:B}, \eqref{eq:A}.

First, we estimate the contribution of the error term $E_1$ as follows
\begin{multline*}
l^{1/2+\epsilon}\sum_{k}h\left(\frac{4k}{K} \right)\int_{0}^{\pi/2}\sqrt{\theta\sin{\theta}}|E_1|d\theta\ll
l^{1/2+\epsilon}\sum_{k}h\left(\frac{4k}{K} \right)\times \\ \biggl(\int_{0}^{c/N}\frac{\theta^3d\theta}{N^{2}}+\int_{c/N}^{\pi/2}\frac{\theta^{3/2}d\theta}{N^{7/2}}\biggr)\ll l^{1/2+\epsilon}
\sum_{k}h\left(\frac{4k}{K} \right)\frac{1}{k^{7/2}}\ll  K\frac{l^{1/2+\epsilon}}{K^{7/2}}.
\end{multline*}
For the main terms the largest contribution comes from the first summand, namely
\begin{equation*}MT:=l^{1/2+\epsilon} \int_{0}^{\pi/2} \left|\sum_{k}h\left(\frac{4k}{K} \right)
J_0(N\theta)\right|\sqrt{\theta \sin{\theta}}d\theta.
\end{equation*}
Indeed, two other summands have similar oscillation but they are smaller is terms of absolute value because
\begin{equation*}
\theta B_0(\theta)=O(\theta), \quad A_1(\theta)=O(\theta^2).
\end{equation*}

Note that the oscillation in $MT$ is only possible  when $N\theta \gg 1$. Hence let us split the integral over $\theta$ into two parts at the point $t:=C/K$ for some absolute constant $C>1$. The first part is bounded by
\begin{equation*}M_1:=l^{1/2+\epsilon}\sum_{k}h\left(\frac{4k}{K} \right)\int_{0}^{t}
\theta d\theta\ll K\frac{l^{1/2+\epsilon}}{K^2}.
\end{equation*}

Now we estimate the second part
\begin{equation*}M_2:=l^{1/2+\epsilon}\int_{t}^{\pi/2}\left|\sum_{k}h\left(\frac{4k}{K} \right)
J_0(N\theta)\right|\sqrt{\theta \sin{\theta}}d\theta.
\end{equation*}
For the $J$-Bessel function we apply representation \eqref{eq:repJBessel}. For $d \geq 1$ the contribution of $R_1$, $R_2$ is majorated by
\begin{multline*}
M_{2,1}:=l^{1/2+\epsilon}\sum_{k}h\left(\frac{4k}{K} \right)\frac{1}{k^{1/2+2d}}\int_{t}^{\pi/2}
\frac{\theta}{\theta^{1/2+2d}}d\theta
\ll \\ l^{1/2+\epsilon} \frac{K}{K^{1/2+2d}}\frac{t^2}{t^{1/2+2d}}\ll l^{1/2+\epsilon}\frac{K}{K^{2}}.
\end{multline*}

Since $N\theta\gg 1$ it is sufficient to consider only the contribution of the first summand in \eqref{eq:repJBessel}, which is bounded by
\begin{equation*}
M_{2,2}:=l^{1/2+\epsilon}\int_{t}^{\pi/2}\sqrt{\theta}\left|\sum_{k}\frac{h\left(4k/K \right)}{\sqrt{2k-1/2}}\cos\left((2k-1/2)\theta-\frac{\pi}{4}\right)\right|d\theta.
\end{equation*}
Other summands in \eqref{eq:repJBessel} have similar oscillation but are smaller in absolute value.
By Poisson's summation formula the sum over $k$ is majorated by a linear combination of expressions of the form
\begin{multline*}
\sum_{u \in\Z}\int_{-\infty}^{\infty}
h\left( \frac{4y}{K}\right)\frac{\exp(i2y\theta)}{\sqrt{y}}\exp(-2\pi iyu)dy=\\
\frac{K}{\sqrt{K}}\sum_{u \in\Z}\int_{-\infty}^{\infty}
\frac{h(4y)}{\sqrt{y}}\exp\left(iyK(2\theta-2\pi u)\right)dy.
\end{multline*}
Since $0<\theta\leq \pi/2$ one has $|2\theta-2\pi u|\gg u$ for $u\neq 0$. Integrating by parts $a\geq 2$ times, we obtain
\begin{multline*}
M_{2,2}\ll l^{1/2+\epsilon}\int_{t}^{\pi/2}\sqrt{\theta} \frac{K}{\sqrt{K}}\left( \frac{1}{(\theta K)^a}+\sum_{u \neq 0}\frac{1}{(Ku)^a}\right)\times \\
\int_{-\infty}^{+\infty}\left|\frac{\partial^a}{\partial y^a}
\left( \frac{h(y)}{\sqrt{y}}\right)\right|dy d\theta.
\end{multline*}
It follows from the definition of function $h$ that
\begin{equation*}
\int_{-\infty}^{+\infty}\left|\frac{\partial^a}{\partial y^a}
\left( \frac{h(y)}{\sqrt{y}}\right)\right|dy\ll 1.
\end{equation*}
Thus
\begin{equation*}
M_{2,2}\ll l^{1/2+\epsilon}\frac{K}{\sqrt{K}}\int_{t}^{\pi/2}\sqrt{\theta} \frac{d\theta}{(\theta K)^a}\ll l^{1/2+\epsilon}\frac{K}{K^{a+1/2}}\frac{t^{3/2}}{t^a}\ll K\frac{l^{1/2+\epsilon}}{K^2}.
\end{equation*}
Finally,
\begin{equation*}
MT\ll M_1+M_2\ll M_1+M_{2,1}+M_{2,2}\ll K\frac{l^{1/2+\epsilon}}{K^2}.
\end{equation*}
\end{proof}

\section*{Acknowledgments}
The authors thank the referee for careful reading of the manuscript and for recommending several improvements in exposition.

\nocite{}

\end{document}